\documentclass[12pt, a4paper, oneside]{article}
\usepackage[top=1in, bottom=1.1in, left=0.9in, right=0.8in]{geometry}
\usepackage[parfill]{parskip}
\usepackage{amsmath}
\usepackage{amsthm}
\usepackage{graphicx}
\usepackage{amssymb}
\usepackage{epstopdf}
\usepackage{setspace}
\usepackage{authblk}
\usepackage{enumerate}
\usepackage{lmodern}
\usepackage[hypcap]{caption}
\usepackage{color}
\definecolor{dblue}{rgb}{0,0,0.45}
\definecolor{red}{rgb}{0.7,0,0}

\usepackage[english]{babel}

\newlength\longest

\newcommand{\cM}{{\mathcal M}}

\begin{document}

\newtheorem{thm}{Theorem}[section]
\newtheorem{lem}[thm]{Lemma}
\newtheorem{prop}[thm]{Proposition}
\newtheorem{cry}[thm]{Corollary}
\theoremstyle{definition}
\newtheorem{dfn}[thm]{Definition}
\newtheorem{example}[thm]{Example}
\theoremstyle{remark}
\newtheorem{rem}[thm]{Remark}
\numberwithin{equation}{section}

\title{A Note on Inclusions of Discrete Morrey Spaces}

\author{Hendra Gunawan${}^{1,*}$, Denny I. Hakim${}^{2,\dagger}$, and Mochammad Idris${}^{3}$}
\affil{${}^{1, 2}$Analysis and Geometry Group,\\Faculty of Mathematics and Natural Sciences,\\
Bandung Institute of Technology,\\Bandung 40132, INDONESIA\\
$^{*}$Email: hgunawan@math.itb.ac.id
\\
$^{\dagger}$Email: dennyivanalhakim@gmail.com

\bigskip

${}^3$Department of Mathematics, \\Faculty of Mathematics and Natural Sciences,\\
Universitas Lambung Mangkurat,\\Banjarbaru 70714, INDONESIA
\\
Email: idemath@gmail.com 
}
\date{}

\maketitle

\begin{abstract}
We  give a necessary condition for inclusion relations between discrete Morrey spaces which can be seen as a complement
of the results in \cite{GKS,HS2}. We also prove another inclusion property of discrete Morrey spaces which can be viewed as
a generalization of the inclusion property of the spaces of $p$-summable sequences. Analogous results for weak type
discrete Morrey spaces is also presented. In addition, we show that each of these inclusion relations is proper. Some
connections between inclusion properties of discrete Morrey spaces and those of Morrey spaces are also discussed.

\bigskip


\end{abstract}

\section{Introduction}

Let $1\le p\le q<\infty$. The discrete Morrey space $\ell^p_q=\ell^p_q(\mathbb{Z})$,
introduced in \cite{GKS},
is defined to be the set of all sequences $x=(x_j)_{j\in \mathbb{Z}}$ such that
\[
\sup_{ m\in \mathbb{Z}, N\in \mathbb{N}_0}
|S_{m,N}|^{\frac{p}{q}-1}
\sum_{j\in S_{m,N}} |x_j|^p<\infty,
\]
where $\mathbb{N}_0:=\mathbb{N}\cup \{0\}$,
$S_{m,N}:=\{ m-N, \ldots, m, \ldots, m+N  \}$, and $|S_{m,N}|=2N+1$.
This space is a Banach space with the norm
\begin{equation}\label{eq:170311-1}
\|x\|_{\ell^p_q}:=
\sup_{ m\in \mathbb{Z}, N\in \mathbb{N}_0}
|S_{m,N}|^{\frac{1}{q}-\frac{1}{p}}
\left(\sum_{j\in S_{m,N}} |x_j|^p\right)^{\frac{1}{p}}.
\end{equation}
We remark that, for $p=q$, we have $\ell^p_p=\ell^p$. Moreover, it is shown in \cite{GKS} that $\ell^p$ is
a proper subset of $\ell^p_q$ whenever $1\le p<q<\infty$. In the same paper, the authors also prove that if
$1\le p_1\le p_2\le q$, then
\begin{equation}\label{eq:170320-1}
\ell^{p_2}_q \subseteq \ell^{p_1}_q.
\end{equation}
In addition, it is shown that
\[
\ell^{p_2}_{q_2} \subseteq \ell^{p_1}_{q_1}
\implies
q_2\le q_1
\]
whenever $1\le p_1\le q_1<\infty$, $1\le p_2\le q_2<\infty$, and $p_1\le p_2$.
Recently, these inclusion results are extended to discrete Morrey spaces on $\mathbb{Z}^d$ in \cite{HS2}. 
Moreover, the authors in \cite{HS2} give a necessary condition for inclusion relations of discrete Morrey spaces on $\mathbb{Z}^d$. 
They also prove that the embedding \eqref{eq:170320-1} is never compact.

One of the aims of this present paper is to give an alternative proof of a necessary condition of inclusion of discrete Morrey spaces. In particular, we
show that $p_1\le p_2$ is necessary for
the inclusion relation \eqref{eq:170320-1}. 
Besides \eqref{eq:170320-1}, we also discuss another inclusion property of discrete Morrey spaces which generalizes
\begin{equation}\label{eq:171114}
\ell^{p_1} \subseteq \ell^{p_2}
\end{equation}
for every $p_1\le p_2$, namely
$
\ell^{p_1}_{q_1 } \subseteq \ell^{p_2}_{q_2 }
$
whenever $p_1 \le  p_2$ and $\frac{  q_1}{q_2 } \le  \frac{p_1}{p_2 }$.
We prove this result by using the fact that $\ell^\infty$ contains $\ell^{p}_q$ for all $1\le p\le q<\infty$. 

We also see later that $\ell^\infty$ coincides with the space $\ell^{p}_{\infty}$ where $p<\infty$.
Here, $\ell^{p}_{\infty}$ is defined to be the set of all sequences $x=(x_j)_{j\in \mathbb{Z}}$ for which
\[
\|x\|_{\ell^p_\infty}
:=
\sup_{m\in \mathbb{Z}, N\in \mathbb{N}_0}
|S_{m,N}|^{-\frac1p}
\left( \sum_{m\in \mathbb{Z}, N\in \mathbb{N}_0}
|x_j|^p
\right)^{\frac1p}
<\infty.
\]
In addition to the results for discrete Morrey spaces, we give the corresponding results for
weak type discrete Morrey spaces in Section \ref{s22}.
The proper inclusion relation between a discrete Morrey space and its weak type is given in Section \ref{s23}.

In Section \ref{s3}, we discuss the relation between discrete Morrey spaces and `continuous' Morrey spaces
on $\mathbb{R}$. Recall that, for $1\le p\le q<\infty$, the Morrey space $\mathcal{M}^p_q(\mathbb{R}^n)$
is defined to be the set of all functions $f\in L^p_{\rm loc}(\mathbb{R}^n)$ such that the norm
\[
\|f\|_{\mathcal{M}^p_q}
:=\sup_{a\in \mathbb{R}^n, r>0}
|B(a,r)|^{\frac{1}{q}-\frac{1}{p}}
\left(
\int_{B(a,r)}
|f(y)|^p \ dy
\right)^{\frac1p}
\]
is finite.
We apply the result in \cite{KS} to recover the inclusion properties of discrete Morrey spaces
from those of Morrey spaces. We also reprove some necessary condition for the inclusion relations between
Morrey spaces in \cite{GHI} by combining the results in Section \ref{s2} and \cite{KS}. Our main result in
this section is a necessary condition for the proper inclusion relation between weak type discrete Morrey spaces
and discrete Morrey spaces.

Let us mention some previous works related to this paper. The inclusion properties of discrete Morrey spaces,
their weak type spaces and their generalization are initially studied in \cite{GKS}. The analogous results on
Morrey spaces $\cM^p_q(\mathbb{R}^n)$ can be found in \cite{GHLM,GHI,HS,P,R,ST}.
The boundedness of the Hardy-Littlewood maximal operator on discrete Morrey spaces is investigated in \cite{GS}.

Throughout this paper, we denote by $C$ a positive constant which is independent of the sequence $x=(x_j)_{j\in
\mathbb{Z}}$ and its value may be different from line to line. We write $A\lesssim B$ if there exists a positive
constant $C$ such that $A\le CB$. Meanwhile, $A\gtrsim B$ means $B\lesssim A$. In addition, we denote by $A\sim B$
if $A\lesssim B$ and $B\lesssim A$.

\section{Main results}\label{s2}

In this section, we consider three types of inclusion properties of discrete Morrey spaces.

\subsection{Inclusion property of discrete Morrey spaces}

Our first result is a necessary condition for the inclusion property of the first kind.

\bigskip

\begin{thm}\label{t1}
Let $1\le p_1\le q<\infty$ and $1\le p_2\le q<\infty$.
If $\ell^{p_2}_q\subseteq \ell^{p_1}_q$, then $p_1\le p_2$.
\end{thm}

\bigskip

\begin{rem}
As mentioned before,  a similar result for discrete Morrey spaces on
$\mathbb{Z}^d$ can be found in \cite{HS2}.
Here, we give a different proof of the necessary condition for this inclusion property.
\end{rem}

\begin{proof}
Assume to the contrary that $1\le p_2<p_1 <q$.
Choose $v, w \in \mathbb{N}$ such that
\begin{equation}\label{eq:140918-1}
\left( \frac{q}{p_1}-1 \right)w +\frac{2q}{p_1}<v< \left( \frac{q}{p_2}-1 \right)w +2.
\end{equation}
Let $k_0$  be the smallest positive integer such that $\displaystyle 1- \frac{1}{2^{2k_0}}>  \frac{1}{2^{ v+w-1 }}$.
Define
$x=(x_j)_{j \in \mathbb{Z}}$ by
\begin{equation}\label{eq:2}
x_j:=
\begin{cases}
1, &\quad |j|=0, 1, 2, \ldots, 2^{v+w}
\\
1, &\quad |j|=2^{k(v+w)}, 2^{k(v+w)}- 2^{kw}, 2^{k(v+w)}-2 (2^{kw}), 2^{k(v+w)}-3(2^{kw}),\\
&\quad \,\,\,\,\,\, \,\,\,\,\,\, \,\,\,\,\,\,\,\,  \ldots,  2^{k(v+w)}- (2^{k(v-2)})(2^{kw}), \ {\rm where }\
k \in \mathbb{N}\cap [k_0, \infty)
\\
0, &\quad  {\rm otherwise}.
\end{cases}
\end{equation}
Let $n\in \mathbb{N} \cap [k_0, \infty)$.
Observe that, for every $1\le p<\infty$ we have
\begin{equation}\label{eq:140918-1a}
\sum_{|i| \le 2^{n (v+w)}}  |x_i|^p = \left(1+2^{ v+w+1} + 2\sum_{k=k_0}^n (1+2^{k(v-2)})\right) \sim  2^{n(v-2)}
\end{equation}
Therefore, for  $p=p_1$, we have
\begin{align}\label{eq:140918-2}
|S_{0, 2^{n(v+w)}}|^{\frac{1}{q}-\frac{1}{p_1}}
\left(
\sum_{j\in S_{0, 2^{n(v+w)}}} |x_j|^{p_1}
\right)^{\frac{1}{p_1}}
&\sim
(2\cdot 2^{n(v+w)}+1)^{\frac{1}{q}-\frac{1}{p_1}} 2^{ \frac{n(v-2)}{p_1} }
\nonumber
\\
&\ge
(3 \cdot 2^{n(v+w)})^{\frac{1}{q}-\frac{1}{p_1}} 2^{ \frac{n(v-2)}{p_1} }
\nonumber
\\
&
=3^{\frac{1}{q}-\frac{1}{p_1}}  2^{n\left( \frac{v+w}{q}-\frac{w}{p_1}-\frac{2}{p_1} \right)}.
\end{align}
According to \eqref{eq:140918-1}, we see that $\displaystyle \frac{(v+w)}{q}-\frac{w}{p_1}-\frac{2}{p_1}>0$,
so that $2^{n\left( \frac{(v+w)}{q}-\frac{w}{p_1}-\frac{2}{p_1} \right)} \to \infty$ as $n \rightarrow \infty$.
Combining this with
\eqref{eq:140918-2}, we have
$$
\|x\|_{\ell^{p_1}_q} \gtrsim \sup_{n\in \mathbb{N}} 2^{n\left( \frac{(v+w)}{q}-\frac{w}{p_1}-\frac{2}{p_1} \right)} = \infty,
$$
which tells us that $x\notin \ell^{p_1}_q$.

We shall show that $x\in \ell^{p_2}_q$.
For every $n\in \mathbb{N} \cap [k_0, \infty)$, define $m_n:=2^{n(v+w)}-2^{n(v+w-2)-1}$ and let
$N_n\in \{0, 1, 2,  \ldots, 2^{n(v+w-2)-1} \}$. We also define
\[
\ell_n:=\max \{ \ell \in \{0, 1, 2, \ldots, 2^{n(v-2)-1}\}: 2^{nw} \ell \le N_n \}.
\]
Observe that
\begin{align}
|S_{m_n, N_n}|^{\frac{1}{q}-\frac{1}{p_2}}
\left(
\sum_{j\in S_{m_n, N_n}} |x_j|^{p_2}
\right)^{\frac{1}{p_2}}
&=(2N_n+1)^{\frac{1}{q}-\frac{1}{p_2}}
\left(
1+2 \ell_n
\right)^{\frac{1}{p_2}}
\nonumber
\\
&\lesssim
(2^{nw} \ell_n)^{\frac{1}{q}-\frac{1}{p_2}} (4\ell_n)^{\frac{1}{p_2}}
\nonumber
\\
&= 4^{\frac{1}{p_2}}
2^{\frac{nw}{q}-\frac{nw}{p_2}} (\ell_n)^{\frac{1}{q}}
\lesssim 2^{n\left( \frac{v+w-2}{q}-\frac{w}{p_2}  \right)}.
\end{align}
Since  $\frac{(v+w-2)}{q}-\frac{w}{p_2}<0$, we have
\begin{equation}\label{eq:3719-1}
|S_{m_n, N_n}|^{\frac{1}{q}-\frac{1}{p_2}}
\left(
\sum_{j\in S_{m_n, N_n}} |x_j|^{p_2}
\right)^{\frac{1}{p_2}}
\lesssim 1.
\end{equation}
If $N\le 2^{k_0(v+w)}$, then
\begin{equation}\label{eq:300918-1}
|S_{0, N}|^{\frac{1}{q}-\frac{1}{p_2}}
\left(
\sum_{j\in S_{0, N}} |x_j|^{p_2}
\right)^{\frac{1}{p_2}}
\le
\left(
\sum_{j\in S_{0, 2^{k_0(v+w)}}} |x_j|^{p_2}
\right)^{\frac{1}{p_2}}
=C2^{\frac{k_0(v-2)}{p_2}}.
\end{equation}
On the other hand, if $N>2^{k_0(v+w)}$, then there exists $n\in \mathbb{N}$ such that
$$
2^{n(v+w)}<N<2^{(n+1)(v+w)}.
$$
Hence we obtain
\begin{align}\label{eq:300918-2}
|S_{0,N}|^{\frac{1}{q}-\frac{1}{p_2}}
\left(
\sum_{j\in S_{0,N}}|x_j|^{p_2}
\right)^{\frac{1}{p_2}}
&\le
|S_{0, 2^{n(v+w)}}|^{\frac{1}{q}-\frac{1}{p_2}}
\left(
\sum_{j\in S_{0,2^{(n+1)(v+w)}}}|x_j|^{p_2}
\right)^{\frac{1}{p_2}}
\nonumber
\\
&\le C \left( 2^{n(v+w)} +1 \right)^{\frac{1}{q}-\frac{1}{p_2}}
2^{\frac{(n+1)(v-2)}{p_2}}
\nonumber
\\
& \le  C 2^{n\left( \frac{w}{q}-\frac{w}{p_2}+\frac{v }{q}-\frac{2 }{p_2} \right)}.
\end{align}
Note that \eqref{eq:140918-1} implies $\displaystyle \frac{w}{q}-\frac{w}{p_2}+\frac{v }{q}-\frac{2 }{p_2}<0$.
As a consequence of this inequality and \eqref{eq:300918-2}, we have
\begin{equation}\label{eq:300918-3}
|S_{0,N}|^{\frac{1}{q}-\frac{1}{p_2}}
\left(
\sum_{j\in S_{0,N}}|x_j|^{p_2}
\right)^{\frac{1}{p_2}}
\lesssim 1.
\end{equation}
Combining \eqref{eq:3719-1} and
\eqref{eq:300918-3}, we get
\[
\sup_{m\in \mathbb{Z}, N\in \mathbb{N}_0}
|S_{m, N}|^{\frac{1}{q}-\frac{1}{p_2}}
\left(
\sum_{j\in S_{m, N}} |x_j|^{p_2}
\right)^{\frac{1}{p_2}} \lesssim 1,
\]
which means that $\|x\|_{\ell^{p_2}_{q}}<\infty$.
Hence, $x\in \ell^{p_2}_q$ but $x\notin \ell^{p_1}_q$.
This contradicts $\ell^{p_2}_q\subseteq \ell^{p_1}_q$.
\end{proof}
\begin{rem}
Note that the condition $p_1 \le p_2$ and $q_2 \le q_1$ is not a necessary condition for the inclusion
$\ell^{p_2}_{q_2} \subseteq \ell^{p_1}_{q_1}$ because
$\ell^{p_2}_{p_2} \subseteq \ell^{p_1}_{p_1}$ when $p_2<p_1$.
\end{rem}

We now  move on to the inclusion property of the second kind.
As a preparation, we show that  $\ell^\infty$ contains $\ell^p_q$ where $1\le p\le q<\infty.$
Recall that $\ell^\infty = \ell^\infty  ( \mathbb{Z} ) $ is the set of all sequences  $x=(x_j)_{j\in \mathbb{Z}}$ such that
\[
\| x \|_{\ell^\infty }:=\sup_{j \in \mathbb{Z} } |x_j|< \infty.
\]

\begin{thm}\label{t1b}
If  $1\le p \le q <\infty$, then
(i) $\ell^{p}_{q } \subset  \ell^{\infty}$ and
(ii) $\ell^{p}_{\infty } = \ell^{\infty}.$
Moreover, the inclusion $  \ell^{p}_{q } \subset  \ell^{\infty}$ is proper.
\end{thm}

\bigskip

\begin{rem}
The inclusion $\ell^{p}_{q } \subset  \ell^{\infty}$ is obtained in \cite{KS}.
For the reader's convenience, we provide our proof and also show that the inclusion is proper.
\end{rem}

\begin{proof}
(i) Suppose that $1\le p\le q<\infty$. Take an arbitrary  $x \in \ell^{p}_{q }$.
Note that for every $m \in \mathbb{Z}$, we have
$$
\|   x \| _{\ell^{p}_{q } } \ge |S_{m,0}|^{\frac{1}{q}-\frac{1}{p}} |x_m|.
$$
Because $|S_{m,0}|=1$, we obtain
$ \|  x \|_{\ell^{\infty }  }=\sup_{j \in \mathbb{Z} } |x_j| \leq \|  x \|_{\ell^{p}_{q } }$.

Next we shall show that $ \ell^\infty \setminus   \ell^p_q \neq \emptyset$.
Define $(y_j)_{j\in \mathbb{Z}}$ by  $y_j :=1$ for every
$ j \in \mathbb{Z} $. We obtain
$$\| y \|_{\ell^\infty }=\sup_{j \in \mathbb{Z} } |y_j|=1< \infty.$$
On the other hand,
$$
\| y \|_{\ell^p_q } =\sup_{m \in \mathbb{Z}, N \in \mathbb{N}}|S_{m,N}|^{\frac{1}{q}-\frac{1}{p}}
\left( \sum_{j \in S_{m,N}}|y_j |^{p}\right)^{\frac{1}{p}} =
\sup_{N \in \mathbb{N}_0}(2N+1)^{\frac{1}{q}-\frac{1}{p}} (2N+1)^{\frac{1}{p}}
= \sup_{N \in \mathbb{N}_0} (2N+1)^{\frac{1}{q}}
= \infty.$$
Thus, $y\in \ell^\infty \setminus \ell^{p}_{q}$.
\bigskip

(ii) Now, take an arbitrary  $x \in \ell^{p}_{\infty }$. For every $m \in \mathbb{Z}$,
$\| x \|_{\ell^{p}_{\infty } } \ge |S_{m,0}|^{ -\frac{1}{p}} |x_m| $ holds. Since $|S_{m,0}|=1$,
we have $$ \|  x \|  _{\ell^{\infty }  }=\sup_{j \in \mathbb{Z} } |x_j| \leq \|  x \|_{\ell^{p}_{\infty } }.$$
Meanwhile, take an arbitrary $x \in \ell^{\infty} $. Observe that
\begin{align*}
\| x \|_{\ell^{p}_{\infty } } &= \sup_{m \in \mathbb{Z}, N \in \mathbb{N}_0}|S_{m,N}|^{-\frac{1}{p}}
\left( \sum_{j \in S_{m,N}}|x_j |^{p}\right)^{\frac{1}{p}} \\
&\le \sup_{m \in \mathbb{Z}, N \in \mathbb{N}_0}|S_{m,N}|^{ -\frac{1}{p}} |S_{m,N}|^{\frac{1}{p}} \| x \|_{\ell^\infty}\\
&= \| x \|_{\ell^\infty}.
\end{align*}
Therefore, $x\in \ell^{p}_{\infty }$.
Hence,
$\ell^{p}_{\infty } = \ell^{\infty}.$
\end{proof}

Using the relation between $\ell^{p}_{q}$ and $\ell^\infty$ in Theorem \ref{t1b}, we shall
prove the inclusion property of discrete Morrey spaces of the second type.

\bigskip

\begin{thm}\label{t1c}
Let $1\le p_1 \le q_1 <\infty$ and $1\le p_2 \le q_2 <\infty$.
If  $p_1 \le  p_2$ and $\frac{q_1}{q_2} \le  \frac{p_1}{p_2}$, then
$
\  \ell^{p_1}_{q_1 } \subseteq \ell^{p_2}_{q_2 }
$
 with
$$
\|  \cdot \|  _{\ell^{p_2}_{q_2 }  } \leq \|  \cdot \| _{\ell^{p_1}_{q_1 } }.
$$
Moreover, $ \ell^{p_2}_{q_2} \setminus   \ell^{p_1}_{q_1} \neq \emptyset$ whenever
($p_1 <  p_2$  and $\frac{  q_1}{q_2 } \le  \frac{p_1}{p_2 }$) or ($p_1 \le   p_2$
and $\frac{  q_1}{q_2 } <  \frac{p_1}{p_2 }$).
\end{thm}

\begin{proof}
Take an arbitrary $x \in \ell^{p_1}_{q_1 }$. Then for every $m\in \mathbb{Z}$ and $N\in \mathbb{N}_0$, we have
\begin{align*}
|S_{m.N}|^{\frac{1}{q_2}-\frac{1}{p_2}}
\left( \sum_{j\in S_{m,N}}
|x_j|^{p_2}
\right)^{\frac{1}{p_2}}
&\le
(2N+1)^{\frac{1}{q_2}-\frac{1}{p_2}}
\left( \sum_{j\in S_{m,N}}
|x_j|^{p_1}
\right)^{\frac{1}{p_2}}
\|x\|_{\ell^\infty}^{1-\frac{p_1}{p_2}}
\\
&\le
(2N+1)^{\frac{1}{q_2}-\frac{1}{p_2}}
\left( \sum_{j\in S_{m,N}}
|x_j|^{p_1}
\right)^{\frac{1}{p_2}}
\|x\|_{\ell^{p_1}_{q_1}}^{1-\frac{p_1}{p_2}}
\\
&\le
(2N+1)^{\frac{1}{q_2}-\frac{p_1}{q_1p_2}}
\|x\|_{\ell^{p_1}_{q_1}},
\end{align*}
where we use Theorem \ref{t1b} and the definition of $\ell^{p_1}_{q_1}$.
Since $\frac{1}{q_2}-\frac{p_1}{q_1p_2}\le 0$, we get
\[
S_{m.N}|^{\frac{1}{q_2}-\frac{1}{p_2}}
\left( \sum_{j\in S_{m,N}}
|x_j|^{p_2}
\right)^{\frac{1}{p_2}}
\le \|x\|_{\ell^{p_1}_{q_1}},
\]
and hence $x\in \ell^{p_2}_{q_2}$ with
\[
\|x\|_{\ell^{p_2}_{q_2}}\le \|x\|_{\ell^{p_1}_{q_1}}.
\]

We now prove the second part of the theorem. Note that our assumption implies $q_1<q_2$.
Define  $(y_j)_{j\in \mathbb{Z}}$ where $y_0 :=1$ dan $y_j  :=|j|^{-\frac{1}{q_2}}$ for every
$ j \in \mathbb{Z} - \{0\} $.
Then, for every $N \in \mathbb{N}_0$ and for $k=1,2$, we have
\[
\sum_{j\in S_{0,N}}|y_j|^{p_k}=\sum_{j=-N}^{N}|y_j|^{p_k} =   1 + 2 \sum_{j=1}^{N}j^{ -\frac{p_k}{q_2}}
\approx 1+  N ^{1-\frac{p_k}{q_2}}.
\]
Therefore
\begin{align*}
\| y \| _{\ell^{p_1}_{q_1 } } &\ge \sup_{N \in \mathbb{N}_0} (1+2N)^{\frac{1}{q_1}-\frac{1}{p_1}}
\left(\sum_{j \in S_{0,N}} |y_j |^{p_1 } \right)^{\frac{1}{p_1}} \\
&\ge C \sup_{N \in \mathbb{N}_0} (1+2N)^{\frac{1}{q_1}-\frac{1}{p_1}} (1+N^{1-\frac{p_1}{q_2}})^{\frac{1}{p_1}} \\
&\ge C \sup_{N\in \mathbb{N}_0}
N^{\frac{1}{q_1}-\frac{1}{q_2}}
= \infty.
\end{align*}
Meanwhile, we obtain
\begin{align*}
\| y \| _{\ell^{p_2}_{q_2 } } &= \sup_{m \in \mathbb{Z}, N \in \mathbb{N}_0} (1+2N)^{\frac{1}{q_2}-\frac{1}{p_2}}
(\sum_{j \in S_{m,N}} |y_j |^{p_2 } )^{\frac{1}{p_2}} \\
&\le  C \sup_{N \in \mathbb{N}_0} (1+2N)^{\frac{1}{q_2}-\frac{1}{p_2}} (1+2N^{1-\frac{p_2}{q_2}})^{\frac{1}{p_2}} \\
&\le C<\infty.
\end{align*}
Thus, $y\in \ell^{p_2}_{q_2} \setminus \ell^{p_1}_{q_1}$. This completes the proof.
\end{proof}

As a corollary of Theorem \ref{t1c}, we obtain the following result.

\bigskip

\begin{thm}\label{t1a}
If  $1\le p \le q_1<\infty$ and $1\le p \le q_2<\infty$ then
 $q_1\le q_2$  if and only if  $\ell^{p }_{q_1}\subseteq \ell^{p }_{q_2}$ with
 $\|  \cdot \|  _{\ell^{p }_{q_2} } \leq \|  \cdot \| _{\ell^{p}_{q_1} }$.
Furthermore, $\ell^{p }_{q_1}$ is a proper subset of $\ell^{p }_{q_2}$ whenever $1\le p \le q_1<  q_2<\infty$.
\end{thm}

\begin{proof}
In view of Theorem \ref{t1c}, we only need to prove that
the inclusion $\ell^{p }_{q_1}\subseteq \ell^{p }_{q_2}$ with $ \|  \cdot \|  _{\ell^{p }_{q_2} } \leq
\|  \cdot \| _{\ell^{p}_{q_1} }$ imply $q_1\le q_2$.
Let $K\in \mathbb{N}$ and 
take an example $x=(x_j)_{j \in \mathbb{Z}}$ where $x_j:=1$ for $|j|\le K$ and $x_j:=0$ for $|j|>K$.
Observe that
$$
(2 K+1)^{\frac{1}{q_2}}=\|  x \|_{\ell^{p }_{q_2} } \leq \|  x \|_{\ell^{p}_{q_1} } = (2K+1)^{\frac{1}{q_1}}.
$$
Since $2K +1 \ge 1$, we conclude that $q_1 \le q_2$.
\end{proof}

\begin{rem}
The argument in the proof of Theorem \ref{t1a} cannot be applied to `continuous' Morrey spaces
$\mathcal{M}^p_{q_k}(\mathbb{R}^n)$, $k=1, 2$, because the case $|B(a,r)|<1$ occurs in the
formula of the $\mathcal{M}^p_{q_k}$-norm.
\end{rem}

\subsection{Incusion property of weak type discrete Morrey spaces}\label{s22}

Let us recall the definition of weak type discrete Morrey spaces.

\bigskip

\begin{dfn}\label{def1}
Let $1\le p\le q<\infty$. The weak type discrete Morrey spaces $w\ell^p_q$ is defined
to be the set of all sequences $x=(x_j)_{j\in \mathbb{Z}}$ for which the quasi-norm
\begin{equation}\label{eq:14}
\|x\|_{w\ell^p_q}
:=
\sup_{m\in \mathbb{Z}, N\in  \mathbb{N}_0, \gamma>0}
\gamma
|S_{m,N}|^{\frac{1}{q}-\frac{1}{p}}
|\{ j \in S_{m,N}: |x_j|>\gamma  \}|^{\frac{1}{p}}
\end{equation}
is finite.
\end{dfn}

The quasi-norm in Definition \ref{def1} can be rewritten as
\begin{equation}
\|x\|_{w\ell^p_q}
=
\sup_{\gamma>0}
\|1_{\{j\in \mathbb{Z}: |x_j|>\gamma \}}\|_{\ell^p_q},
\end{equation}
where $1_{\{j\in \mathbb{Z}: |x_j|>\gamma \}}:=(1_{\{j\in \mathbb{Z}: |x_j|>\gamma \}}(j))_{j\in \mathbb{Z}}$
and
$$
1_{\{j\in \mathbb{Z}: |x_j|>\gamma \}}(j):=\begin{cases}
1, &|x_j|>\gamma,
\\
0, &|x_j|\le \gamma.
\end{cases}
$$
Note that, for $p=q$, the space $w\ell^p_p$ is the weak type $\ell^p$ space.
It is shown in \cite[Example 3.1]{GKS} that the space $\ell^p$ is a proper subset of $w\ell^p_p$.
More general, the discrete Morrey space $\ell^p_q$ is a subset of $w\ell^p_q$ (see \cite[Theorem 3.2]{GKS}).

For the inclusion between weak type discrete Morrey spaces, it is proven in \cite{GKS} that
$w\ell^{p_2}_q \subseteq w\ell^{p_1}_q$ whenever $1\le p_1\le p_2\le q<\infty$.
Now we show that $p_1\le p_2$ is a necessary condition for these inclusions.

\bigskip

\begin{thm}\label{t2}
	Let $1\le p_1\le q<\infty$ and $1\le p_2\le q<\infty$.
	If $w\ell^{p_2}_q\subseteq w\ell^{p_1}_q$, then $p_1\le p_2$.
\end{thm}

\begin{proof}
Assume to the contrary that $p_1>p_2$.
Let $x=(x_j)_{j\in \mathbb{Z}}$ be defined by \eqref{eq:2}.
Since $x\in \ell^{p_2}_q $ and $ \ell^{p_2}_q \subseteq w\ell^{p_2}_q$, we have
$x\in  w\ell^{p_2}_q$.
If we can prove that $x\notin w\ell^{p_1}_q$, then we obtain a contradiction, so we
must have $p_1 \le p_2$.
Now we show that $x\notin w\ell^{p_1}_q$.
Observe that
\[
\|x\|_{w\ell^{p_1}_q}
\ge
\frac{1}{2}
\sup_{m\in \mathbb{Z}, N\in \mathbb{N}_0}
|S_{m,N}|^{\frac{1}{q}-\frac{1}{p_1}}
\left|\left\{
j\in S_{m,N}:
|x_j|>\frac{1}{2}
  \right\}\right|^{\frac{1}{p_1}}.
\]
Since either $x_j=0$ or $x_j=1$, we have
\begin{align*}
\sup_{m\in \mathbb{Z}, N\in \mathbb{N}_0}
|S_{m,N}|^{\frac{1}{q}-\frac{1}{p_1}}
\left|\left\{
j\in S_{m,N} : |x_j|>\frac{1}{2}
\right\}\right|^{\frac{1}{p_1}}
&=
\sup_{m\in \mathbb{Z}, N\in \mathbb{N}_0}
|S_{m,N}|^{\frac{1}{q}-\frac{1}{p_1}}
\left(
\sum_{j\in S_{m,N}:
	|x_j|>\frac{1}{2}}
|x_j|^{p_1}
\right)^{\frac{1}{p_1}}
\\
&=
\sup_{m\in \mathbb{Z}, N\in \mathbb{N}_0}
|S_{m,N}|^{\frac{1}{q}-\frac{1}{p_1}}
\left(
\sum_{j\in S_{m,N}}
|x_j|^{p_1}
\right)^{\frac{1}{p_1}}
\\
&=
\|x\|_{\ell^{p_1}_q}.
\end{align*}
Consequently,
\[
\|x\|_{w\ell^{p_1}_q}
\ge
\frac12
\|x\|_{\ell^{p_1}_q}
=\infty,
\]
as desired.
\end{proof}

Similar with the (strong type) discrete Morrey spaces, we also show the inclusion property of weak
type discrete Morrey spaces of the second kind.

\bigskip

\begin{thm}\label{t2b}
Let $1\le p_1 \le q_1 <\infty$ and $1\le p_2 \le q_2 <\infty$. If
$p_1 \le  p_2$ and $\frac{  q_1}{q_2 } \le  \frac{p_1}{p_2 }$, then
$
\  w\ell^{p_1}_{q_1 } \subseteq w\ell^{p_2}_{q_2 }
$
with
$$
\|  \cdot \|  _{w\ell^{p_2}_{q_2 }  } \leq \|  \cdot \| _{w\ell^{p_1}_{q_1 } }.
$$
Moreover, $ w \ell^{p_2}_{q_2} \setminus  w \ell^{p_1}_{q_1} \neq \emptyset$ whenever
($p_1 <  p_2$  and $\frac{  q_1}{q_2 } \le  \frac{p_1}{p_2 }$) or ($p_1 \le  p_2$  
and $\frac{  q_1}{q_2 } <  \frac{p_1}{p_2 }$).
\end{thm}

\begin{proof}
Let $x\in w\ell^{p_1}_{q_1}$. According to Theorem \ref{t1c}, we have
\begin{align*}
\|x\|_{w\ell^{p_2}_{q_2}}
=\sup_{\gamma>0}
\gamma \|1_{\{j\in \mathbb{Z}: |x_j|>\gamma \}}\|_{\ell^{p_2}_{q_2}}
\le
\sup_{\gamma>0}
\gamma \|1_{\{j\in \mathbb{Z}: |x_j|>\gamma \}}\|_{\ell^{p_1}_{q_1}}
=\|x\|_{\ell^{p_1}_{q_1}}<\infty,
\end{align*}
which tells us that $x\in w\ell^{p_2}_{q_2}$.
Hence, $w\ell^{p_1}_{q_1}\subseteq w\ell^{p_2}_{q_2}$ with
$\|\cdot\|_{w\ell^{p_2}_{q_2}} \le  \|\cdot\|_{w\ell^{p_1}_{q_1}}$.

Now we prove the second part of this theorem.
By our assumption, we have $q_1<q_2$.
Let $y=(y_j)_{j\in \mathbb{Z}}$ be defined as in the proof of Theorem \ref{t1c}. 
Note that $y\in \ell^{p_2}_{q_2}$. Therefore,
$$ 
\| y \| _{w\ell^{p_2}_{q_2 } }\le \| y \| _{\ell^{p_2}_{q_2 } }<\infty.
$$
Hence, $y\in w\ell^{p_2}_{q_2}$.
On the other hand,
for every $N\in \mathbb{N}$, we have
\begin{align*}
\frac{N^{-\frac{1}{q_2}}}{2}
|S_{0,N}|^{\frac{1}{q_1}-\frac{1}{p_1}}
\left|
\left\{j\in S_{0,N}: |y_j|> \frac{N^{-\frac{1}{q_2}}}{2}  \right\}
\right|^{\frac{1}{p_1}}
= \frac{N^{-\frac{1}{q_2}}}{2} (2N+1)^{\frac{1}{q_1}-\frac{1}{p_1}}
(2N+1)^{\frac{1}{p_1}}
=C N^{\frac{1}{q_1}-\frac{1}{q_2}}.
\end{align*}
Consequently,
\begin{align*}
\| y \| _{w\ell^{p_1}_{q_1 } } \ge & C \sup_{N \in \mathbb{N}}
N^{\frac{1}{q_1}-\frac{1}{q_2}}
= \infty.
\end{align*}
Hence, $y\notin w\ell^{p_1}_{q_1}$. Thus, $ w \ell^{p_2}_{q_2} \setminus  w \ell^{p_1}_{q_1} \neq \emptyset$.
\end{proof}

Necessary and sufficient conditions for the inclusion property of weak type discrete Morrey spaces
of the second kind are presented in the following theorem.

\bigskip

\begin{thm}\label{t2c}
Let $1\le p \le q_1<\infty$ and $1\le p \le q_2<\infty$. Then
$q_1\le q_2$  if and only if  $w\ell^{p }_{q_1}\subseteq w\ell^{p }_{q_2}$ with
$\|  \cdot \|  _{w\ell^{p }_{q_2} } \leq \|  \cdot \| _{w\ell^{p}_{q_1} }$.
Furthermore, $w\ell^{p }_{q_1}$ is a proper subset of $w\ell^{p }_{q_2}$ whenever $ q_1<  q_2$.
\end{thm}

\begin{proof}
If $q_1\le q_2$, then by taking $p_1=p_2=p$ in Theorem \ref{t2b}, we get
$w\ell^{p }_{q_1}\subseteq w\ell^{p }_{q_2}$ with $ \|  \cdot \|  _{w\ell^{p }_{q_2} } 
\leq \|  \cdot \| _{w\ell^{p}_{q_1} }$.
Conversely, assume that $w\ell^{p }_{q_1}\subseteq w\ell^{p }_{q_2}$ with
$\|  \cdot \|  _{w\ell^{p }_{q_2} } \leq \|  \cdot \| _{w\ell^{p}_{q_1} }$.
Let $x=(x_j)_{j\in \mathbb{Z}}$ be defined  by
$$
x_j:=\begin{cases}
1, &|x_j|\le k,
\\
0, &|x_j|>k.
\end{cases}
$$
Then
$$
(2 K+1)^{\frac{1}{q_2}}
=\|x\|_{\ell^{p}_{q_2}}
=\|  x \|  _{w\ell^{p }_{q_2} }
\le \|  x \| _{w\ell^{p}_{q_1} }
=\|  x \| _{\ell^{p}_{q_1} }
=(2 K+1)^{\frac{1}{q_1}}.
$$
Therefore, $(2K+1)^{\frac{1}{q_2}-\frac{1}{q_1}}\le 1$.
Since $2K +1 \ge 1$, we have $\displaystyle \frac{1}{q_2}-\frac{1}{q_1}\le 0$, so that $q_1 \le q_2$.
The second part of this theorem follows from Theorem \ref{t2b} by taking
$p_1=p_2=p$.
\end{proof}

\subsection{Inclusion relation between discrete Morrey spaces and weak type discrete Morrey spaces}\label{s23}

\begin{thm}\label{t8}
Let $1\le p_1<p_2\le q<\infty$. Then the inclusion $w\ell^{p_2}_{q}\subseteq \ell^{p_1}_{q}$ is proper.
\end{thm}

\begin{proof}
Let $x\in w\ell^{p_2}_{q}$.
We shall show that there exists $C>0$ such that
\begin{align}\label{eq:1092018-1}
|S_{m,N}|^{\frac{1}{q}-\frac{1}{p_1}}
\left( \sum_{j\in S_{m,N}}
|x_j|^{p_1}\right)^{\frac{1}{p_1}}
\le C \|x\|_{w\ell^{p_2}_q}
\end{align}
for every $m\in \mathbb{Z}$ and $N\in \mathbb{N}_0$. Observe that
\[
\sum_{j\in S_{m,N}} |x_j|^{p_1}
=p_1 \int_{0}^{\infty} t^{p_1-1} |\{j\in S_{m,N}: |x_j|>t  \}| \ dt
=I_1+I_2,
\]
where
\[
I_1(R):=p_1 \int_{0}^{R} t^{p_1-1} |\{j\in S_{m,N}: |x_j|>t  \}| \ dt, \
I_2(R):=p_1 \int_{R}^{\infty} t^{p_1-1} |\{j\in S_{m,N}: |x_j|>t  \}| \ dt, \
\]
and $R>0$ is chosen later. The estimate for $I_1(R)$ is
\begin{equation}\label{eq:100918-0}
I_1(R)\le p_1\int_{0}^{R} t^{p_1-1} |S_{m,N}| \ dt
=p_1 (2N+1) \int_{0}^{R} t^{p_1-1} \ dt
=(2N+1)R^{p_1}.
\end{equation}
Meanwhile, by using Definition \ref{def1} and $p_1<p_2$, we have
\begin{align}\label{eq:100918-1}
I_2(R)
&\le
p_1 \int_{R}^{\infty}
t^{p_1-1} |S_{m,N}|^{1-\frac{p_2}{q}} t^{-p_2} \|x\|_{w\ell^{p_2}_{q}}^{p_2} \ dt
\nonumber
\\
&=
p_1 (2N+1)^{1-\frac{p_2}{q}}
\|x\|_{w\ell^{p_2}_q}^{p_2}
\int_{R}^{\infty} t^{p_1-p_2-1} \ dt
\nonumber
\\
&=
\frac{p_1}{p_2-p_1}
(2N+1)^{1-\frac{p_2}{q}}
\|x\|_{w\ell^{p_2}_q}^{p_2} R^{p_1-p_2}.
\end{align}
Combining \eqref{eq:100918-0} and \eqref{eq:100918-1} and taking $\displaystyle R :=
\left( \frac{p_1}{p_2-p_1} \right)^{\frac{1}{p_2}}
\frac{\|x\|_{w\ell^{p_2}_q}}{(2N+1)^{\frac{1}{q}}}$, we get
\begin{align}\label{eq:091018-1}
\sum_{j\in S_{m,N}}
|x_j|^{p_1}
\le C (2N+1)^{1-\frac{p_1}{q}}
\|x\|_{w\ell^{p_2}_q}^{p_1}.
\end{align}
Hence, \eqref{eq:1092018-1} follows
by taking $p_1$-th root of \eqref{eq:091018-1} and multiplying by $|S_{m,N}|^{\frac{1}{q}-\frac{1}{p_1}}$.
Since \eqref{eq:1092018-1} holds for arbitrary $m\in \mathbb{Z}$ and
$N\in \mathbb{N}_0$, we have $x\in \ell^{p_1}_q$ with
\[
\|x\|_{\ell^{p_1}_q} \le
C\|x\|_{w\ell^{p_2}_q},
\]
as desired.

We shall now prove that the inclusion is proper. Choose $v, w \in \mathbb{N}$ such that
\begin{align}\label{eq:140918-10}
\left( \frac{q}{p_2}-1 \right)w +\frac{2q}{p_2}<v< \left( \frac{q}{p_1}-1 \right)w +2.
\end{align}
Let $k_0$  be the smallest positive integer such that $\displaystyle 1- \frac{1}{2^{2k_0}}> \frac{1}{2^{ v+w-1 }}$.
Define
$x=(x_j)_{j \in \mathbb{Z}}$ by the formula
\begin{align}\label{eq:20}
x_j:=
\begin{cases}
1, &\quad |j|=0, 1, 2, \ldots, 2^{v+w}
\\
1, &\quad |j|=2^{k(v+w)}, 2^{k(v+w)}- 2^{kw}, 2^{k(v+w)}-2 (2^{kw}), 2^{k(v+w)}-3(2^{kw}),\\
&\quad \,\,\,\,\,\, \,\,\,\,\,\, \,\,\,\,\,\,\,\,  \ldots,  2^{k(v+w)}- (2^{k(v-2)})(2^{kw}), \ {\rm where }\
k \in \mathbb{N}\cap [k_0, \infty)
\\
0, &\quad  {\rm otherwise}.
\end{cases}
\end{align}
By the same calculation as in the proof of Theorem \ref{t1},
we obtain  $x\in \ell^{p_1}_q$ but $x\notin \ell^{p_2}_q$.
Moreover, repeating the calculation in the proof of Theorem \ref{t2}, we obtain $\|x\|_{w\ell^{p_2}_q}\ge
\frac{1}{2} \|x\|_{\ell^{p_2}_q}$, so that
$x\notin w\ell^{p_2}_q$. This shows that $x\in \ell^{p_1}_q\setminus w\ell^{p_2}_q $. 
Thus, we conclude that the inclusion $w\ell^{p_2}_{q}\subseteq \ell^{p_1}_{q}$ is proper.
\end{proof}

\section{Relation between discrete Morrey spaces and their continuous counterpart}\label{s3}

In this section, we discuss the relation between the inclusion property of discrete Morrey spaces and
that of Morrey spaces. In particular, we reprove the inclusion property of discrete Morrey spaces by
using the inclusion property of Morrey spaces, and we recover some necessary conditions for the inclusion
property of Morrey spaces. We also give a necessary condition for the proper inclusion relation between
discrete Morrey spaces and weak type discrete Morrey spaces. Our results are based on some recent results
in \cite{KS}.

We now recall some definitions and notation. Let $1\le p\le q<\infty$. The Morrey space 
$\mathcal{M}^p_q(\mathbb{R})$
is defined to be the set of all functions $f\in L^p_{\rm loc}(\mathbb{R})$ for which
\[
\|f\|_{\mathcal{M}^p_q}:=
\sup_{a\in \mathbb{R}, r>0}
(2r)^{\frac{1}{q}-\frac{1}{p}}
\left(
\int_{a-r}^{a+r}
|f(t)|^p \ dt
\right)^{\frac{1}{p}}
\]
is finite. The weak Morrey space $w\mathcal{M}^p_q(\mathbb{R})$ is defined to be the set of all measurable
functions $f$ on $\mathbb{R}$ such that the quasi-norm
\[
\|f\|_{w\mathcal{M}^p_q}
:=\sup_{\lambda>0}\|\lambda \chi_{\{ x \in \mathbb{R}:|f(x)|>\lambda \}}\|_{\mathcal{M}^p_q}
\]
is finite. From these definitions, it is clear that $\mathcal{M}^p_q(\mathbb{R})\subseteq w\mathcal{M}^p_q(\mathbb{R})$.
The relation between $\ell^p_q$ and $\mathcal{M}^p_q(\mathbb{R})$ is given as follows. For every
sequence $x=(x_j)_{j\in \mathbb{Z}}$, define a function $\overline{x}:\mathbb{R}\to [0,\infty)$ by
\[
\overline{x}(t)
:=\left(
\sum_{j\in \mathbb{Z}}
|x_j|^{p}\chi_{[j,j+1)}(t)
\right)^{\frac1p}.
\]
Then, the discrete Morrey space $ \ell^p_q$ can be realized as a closed subspace of $ \mathcal{M}^p_q(\mathbb{R})$
in the following sense.

\bigskip

\begin{thm}\label{ES-1}{\rm \cite[Theorem 2.1]{KS}}
Let $1\le p\le q<\infty$. Then, there exist positive constants $C_1$ and $C_2$ such that
\begin{equation}\label{eq:ES-0}
\|\overline{x}\|_{\mathcal{M}^p_q}
\le
C_1 \|x\|_{\ell^p_q}
\end{equation}
for every $x\in \ell^p_q$
and
\begin{equation}\label{eq:ES-1}
\|y\|_{\ell^p_q}
\le
C_2\|\overline{y}\|_{\mathcal{M}^p_q}
\end{equation}
for every sequence $y=(y_j)_{j\in \mathbb{Z}}$.
\end{thm}

\bigskip

The analogous result for weak type discrete Morrey spaces is presented in the following.

\bigskip

\begin{thm}\label{ES-2}{\rm \cite[Theorem 2.3]{KS}}
Let $1\le p\le q<\infty$. Then, there exist positive constants $C_1$ and $C_2$ such
that for every $x\in w\ell^p_q$ and for every sequence $y=(y_j)_{j\in \mathbb{Z}}$ the inequalities
	\begin{align}\label{eq:ES-3}
	\|\overline{x}\|_{w\mathcal{M}^p_q}
	\le
	C_1 \|x\|_{w\ell^p_q}
	\end{align}
	and
\begin{align}\label{eq:ES-2}
	\|y\|_{w\ell^p_q}
\le C_2
\|\overline{y}\|_{w\mathcal{M}^p_q}
\end{align}
hold.
\end{thm}

We apply Theorems \ref{ES-1} and \ref{ES-2} with inclusion of Morrey spaces to recover
the inclusion of discrete Morrey spaces obtained in \cite{GKS}.

\bigskip

\begin{cry}{\rm \cite{GKS}}
Let $1\le p_1\le p_2 \le q<\infty$. Then $\ell^{p_2}_{q}\subseteq \ell^{p_1}_{q}$ and
$w\ell^{p_2}_{q} \subseteq w\ell^{p_1}_{q}$.
\end{cry}

\begin{proof}
Our proof is an alternative to that in \cite{GKS}.
Let $x\in \ell^{p_2}_{q}$. Then, it follows from \eqref{eq:ES-0} that $\overline{x}\in \mathcal{M}^{p_2}_{q}$.
Since $p_1\le p_2$, we have $\mathcal{M}^{p_2}_q \subseteq \mathcal{M}^{p_1}_q$, so that
$\overline{x}\in \mathcal{M}^{p_1}_{q}$. According to \eqref{ES-1}, we have
$x\in \ell^{p_1}_q$. Thus, $\ell^{p_2}_q\subseteq \ell^{p_1}_q$. To prove that
$w\ell^{p_2}_{q} \subseteq w\ell^{p_1}_{q}$, let $y\in w\ell^{p_2}_{q}$. Then, as a consequence
of \eqref{eq:ES-3}, we have $\overline{y} \in w\mathcal{M}^{p_2}_q$.
The inclusion $w\mathcal{M}^{p_2}_q \subseteq w\mathcal{M}^{p_1}_q$ implies $\overline{y} \in w\mathcal{M}^{p_1}_q$.
By virtue of \eqref{eq:ES-2}, we have
$y\in w\ell^{p_1}_{q}$. Thus, $w\ell^{p_2}_{q} \subseteq w\ell^{p_1}_{q}$.
\end{proof}

As a consequence of Theorems \ref{t1}, \ref{t2},  \ref{ES-1}, and \ref{ES-2}, we may
recover a necessary condition for the inclusion property of Morrey spaces and that of
weak Morrey spaces obtained in \cite{GHI}.

\bigskip

\begin{cry}{\rm \cite[Theorem 1.6]{GHI}}
Let $1\le p_1\le q<\infty$ and $1\le p_2\le q<\infty$.
\begin{enumerate}
\item If $\mathcal{M}^{p_2}_{q}\subseteq \mathcal{M}^{p_1}_q$, then $p_1\le p_2$.
\item If $w\mathcal{M}^{p_2}_{q}\subseteq w\mathcal{M}^{p_1}_q$, then $p_1\le p_2$.
\end{enumerate}
\end{cry}

\begin{proof}
\

\begin{enumerate}
\item Let $x\in \ell^{p_2}_{q}$. Then, according to Theorem \ref{ES-1}, we have $\overline{x}\in \mathcal{M}^{p_2}_{q}$.
Since $\mathcal{M}^{p_2}_{q}\subseteq \mathcal{M}^{p_1}_q$,
we obtain $\overline{x} \in \mathcal{M}^{p_1}_{q}$.
This fact and \eqref{eq:ES-1} imply $x\in \ell^{p_1}_q$. Therefore, $\ell^{p_2}_q\subseteq \ell^{p_1}_q$.
Consequently, by virtue of Theorem \ref{t1}, we conclude that $p_1\le p_2$.
\item  Let $x\in w\ell^{p_2}_{q}$. By virtue of Theorem \ref{ES-2}, we have $\overline{x}\in
w\mathcal{M}^{p_2}_{q}$, so that $\overline{x} \in w\mathcal{M}^{p_1}_{q}$. Combining this with \eqref{eq:ES-2},
we obtain $x\in w\ell^{p_1}_q$. Consequently, $\ell^{p_2}_q\subseteq \ell^{p_1}_q$.
Thus, the conclusion follows from Theorem \ref{t2}.
\end{enumerate}
\end{proof}

Finally, we present our main result in this section, namely a necessary condition for the proper inclusion
relation between weak type discrete Morrey spaces and discrete Morrey spaces.

\bigskip

\begin{thm}\label{thm:190319}
Let $1\le p_1\le q<\infty$ and $1\le p_2\le q<\infty$.
If the inclusion $w\ell^{p_2}_q \subseteq \ell^{p_1}_q$ is proper, then $p_1<p_2$.
\end{thm}

\bigskip

\begin{rem}
Note that Theorem \ref{thm:190319} can be viewed as a complement to Theorem \ref{t8}.
The analogous result in the case of Morrey spaces can be found in \cite{GHI, GHNS}.
\end{rem}

\begin{proof}[Proof of Theorem \ref{thm:190319}]
Assume to the contrary that $p_1\ge p_2$.
Let $x\in \ell^{p_1}_q \setminus w\ell^{p_2}_q$.
It follows from \eqref{eq:ES-0} that $\overline{x}\in \mathcal{M}^{p_1}_{q}$. Since $p_1\ge p_2$, we have
$\mathcal{M}^{p_1}_q \subseteq \mathcal{M}^{p_2}_q$, so that
$\overline{x}\in \mathcal{M}^{p_2}_{q}$.
Combining this with \eqref{eq:ES-1}, we get $x\in \ell^{p_2}_q$. Consequently, $x\in w\ell^{p_2}_q$.
This contradicts the fact that $x\in \ell^{p_1}_q \setminus w\ell^{p_2}_q$.
Thus, we must have $p_1<p_2$.
\end{proof}

\noindent{\bf Acknowledgement}. The first and second authors are supported by
P3MI--ITB Research and Innovation Program 2018.

\end{document}